\documentclass[11pt,reqno,oneside]{amsart}

\usepackage{graphicx}
\usepackage{amssymb}
\usepackage{epstopdf}

\usepackage[a4paper, total={5.56in, 9in}]{geometry}

\usepackage{amsmath,amsfonts,amsthm,mathrsfs,amssymb,cite}
\usepackage[usenames]{color}

\usepackage{xcolor}

\newtheorem{thm}{Theorem}[section]

\newtheorem{cor}[thm]{Corollary}
\newtheorem{lem}{Lemma}[section]

\theoremstyle{definition}
\newtheorem{defn}{Definition}[section]

\theoremstyle{remark}

\newtheorem{rem}{Remark}[section]
\numberwithin{equation}{section}

\numberwithin{equation}{section}

\newcounter{saveeqn}
\newcommand{\pare}[1]{\left(#1\right)}

\newcommand{\abs}[1]{\left\lvert #1 \right\rvert}

\newcommand{\RR}{\mathbb{R}}

\newcommand{\Ss}{\mathbb{S}^2}
\newcommand{\CC}{\mathbb{C}}
\newcommand{\im}{\mathrm{i}}
\DeclareMathOperator{\curl}{curl}%{\nabla\cros}%

\DeclareMathOperator{\cros}{\wedge}
\DeclareMathOperator{\supp}{supp}

\DeclareMathOperator{\Span}{span}

\newcommand{\Cor}{\mathcal{K}}%Polyhedral corner
\newcommand{\claS}{\mathscr{A}}

\newcommand{\nor}{\nu}%Unit normal vector
\newcommand{\far}[1]{#1_{\infty}}

\newcommand{\EM}[1]{\mathbf{#1}}

\newcommand{\Ep}{\varepsilon}%electric permittivity
\newcommand{\Mp}{\mu}%magnetic permeability
\newcommand{\Eci}{\gamma}

\title[Geometric properties in electromagnetism and applications]{On an electromagnetic problem in a corner and its applications }

\author{Emilia Bl{\aa}sten}
\address{Department of Mathematics and Statistics, University of Helsinki, Finland}
\email{emilia.blasten@helsinki.fi}

\author{Hongyu Liu}
\address{Department of Mathematics, Hong Kong Baptist University, Kowloon, Hong Kong SAR, China}
\email{hongyu.liuip@gmail.com, hongyuliu@hkbu.edu.hk}

\author{Jingni Xiao}
\address{Department of Mathematics, Rutgers University, Piscataway, NJ 08854, USA}
\email{jingni.xiao@rutgers.edu}

\date{} % Activate to display a given date or no date (if empty),
\begin{document}
	\maketitle
	
	\begin{abstract}
		
		Let $\Cor^{r_0}_{x_0}$ be a (non-degenerate) truncated corner in $\mathbb{R}^3$ with $x_0\in\mathbb{R}^3$ being its apex, and $\mathbf{F}_j\in C^\alpha(\overline{\Cor^{r_0}_{x_0}}; \mathbb{C}^3)$, $j=1,2$, where $\alpha$ is the positive H\"older index. Consider the following electromagnetic problem
		\[
		\left\{
		\begin{split}
		& \nabla\wedge  \EM{E}-\im\omega \Mp_0 \EM{H}=\EM{F}_{1} \quad \mbox{in $\Cor^{r_0}_{x_0}$},\\
		& \, \nabla\wedge \EM{H}+\im\omega\Ep_0 \EM{E}=\EM{F}_{2} \quad \mbox{in $\Cor^{r_0}_{x_0}$},\\
		&\, \nor\cros \EM{E}=\nor\cros\EM{H}=0 \hspace*{0.7cm}\mbox{on $\partial \Cor^{r_0}_{x_0}\setminus \partial B_{r_0}(x_0)$},
		\end{split}\right.
		\]
		where
		$\nu$ denotes the exterior unit normal vector of $\partial \Cor^{r_0}_{x_0}$.
		We prove that $\mathbf{F}_1$ and $\mathbf{F}_2$
		must vanish at the apex $x_0$. There are a series of interesting consequences of this vanishing property in
		several separate but intriguingly connected topics in electromagnetism. First, we can geometrically characterize non-radiating sources in 
		time-harmonic electromagnetic scattering. Secondly, we
		consider the inverse source scattering problem for time-harmonic electromagnetic waves
		and establish the uniqueness result in determining the polyhedral support of a source by a single far-field
		measurement. Thirdly, we derive a property of the geometric structure of electromagnetic interior transmission eigenfunctions near corners. Finally, we also discuss its implication to invisibility cloaking and inverse medium scattering.

		\medskip

		
		
	\end{abstract}
	
	\section{Introduction}
	
	In this paper, we are mainly concerned with the time-harmonic Maxwell system. Let $\omega\in\mathbb{R}_+$ denote the frequency and let $\varepsilon_0, \mu_0\in\mathbb{R}_+$, respectively, signify the electric permittivity and magnetic permeability of a
	uniformly homogeneous space. Throughout, we let $\mathbf{E}$ and $\mathbf{H}$ denote, respectively,
	the electric and magnetic fields, which are $\mathbb{C}^3$-valued functions. 
	%Focusing initially on the mathematics but not physics, we introduce a key Maxwell system for our study next.
	We first introduce the notion of the corner and the corresponding Maxwell system locally around the corner.
	
	Let $w_j\in\Ss:=\{x\in\mathbb{R}^3; |x|=1\}$, $j=1,\ldots, n$, be $n$ number of unit vectors with $n\geq 3$ such that they are triple-wise linearly independent. For a given point $x_0\in\mathbb{R}^3$, we define
	\begin{equation}\label{eq:CorN}
		\Cor=\Cor_{w_1,\ldots,w_n;x_0}:= \{x=x_0+\sum_{j=1}^{n} c_j w_j;\ c_j>0,j=1,\ldots,n\}\subset\mathbb{R}^3. 
	\end{equation}
	We assume that $\Cor$ is strictly convex: it is convex and must fit into a spherical cone of opening angle less than $\pi$. We may assume that none of the $w_j$'s are redundant. Then
	$\Cor$ is called a convex polyhedral cone with $n$ edges in $\mathbb{R}^3$. The point $x_0$ is the apex of the cone and $w_j$, $j=1,2,\ldots, n$, are the $n$ directions of the corresponding edges. Given a constant $r_0\in\mathbb{R}_+$, we define the truncated parallelepiped $\Cor^{r_0}=\Cor^{r_0}_{x_0}$ as
	\begin{equation}\label{eq:cone2}
		\Cor^{r_0}=\Cor_{x_0}^{r_0}=\Cor^{r_0}_{w_1,\ldots,w_n;x_0}:=\Cor_{w_1,\ldots,w_n;x_0}\cap B_{r_0}(x_0).
	\end{equation}
	Let $\mathbf{F}_1$ and $\mathbf{F}_2$ be two $\mathbb{C}^3$-valued functions such that $\mathbf{F}_j\in C^\alpha(\overline{\Cor^{r_0}})^3$, $j=1,2$, where the H\"older index $\alpha\in (0, 1)$. We start from the following Maxwell system in the corner $\mathcal{K}_{x_0}^{r_0}$, %as defined in \eqref{eq:cone2}, 
	\begin{equation}\label{eq:MaxwellCor}
		\left\{
		\begin{split}
			& \nabla\wedge  \EM{E}-\im\omega \Mp_0 \EM{H}=\EM{F}_{1} \quad \mbox{in $\Cor_{x_0}^{r_0}$},\\
			& \, \nabla\wedge \EM{H}+\im\omega\Ep_0 \EM{E}=\EM{F}_{2} \quad \mbox{in $\Cor_{x_0}^{r_0}$},\\
			&\, \nor\cros \EM{E}=\nor\cros\EM{H}=0 \hspace*{0.8cm}\mbox{on $\partial \Cor_{x_0}^{r_0}\setminus \partial B_{r_0}(x_0)$}, 
		\end{split}\right.
	\end{equation}
	where $\im:=\sqrt{-1}$ denotes the imaginary unit and $\nu\in\mathbb{S}^2$ is the exterior unit normal vector to $\partial \Cor_{x_0}^{r_0}\setminus \partial B_{r_0}(x_0)$.
	We emphasize that the homogeneous boundary conditions in \eqref{eq:MaxwellCor} are imposed only on the faces of the corner $\mathcal{K}_{x_0}^{r_0}$ around the apex $x_0$. 
	
	It can be seen from later discussion that, the problem \eqref{eq:MaxwellCor} is the key for several separate topics in
	time-harmonic electromagnetic scattering, including the geometric characterization of non-radiating sources and the uniqueness in determining the support of radiating sources. 
	Moreover, mathematical models for invisible medium scatterers, geometric structures of interior transmission eigenfunctions, as well as the unique determination of the support of inhomogeneous medium scatterers, can be also reduced into the local problem \eqref{eq:MaxwellCor}.
	These topics, in particular those associated with medium scattering, have received considerable interest in the literature recently and have been extensively investigated from different perspectives \cite{Bsource,BLY,BL2016,BL2017,BL2017b,BL2018,BPS,BLLW,BV,CaX,DCL,ElH15,ElH,HSV,LLL17,PSV}, but mainly for the Helmholtz system governing the acoustic scattering. In this paper, we aim to extend some of the earlier results %in \cite{Bsource,BL2016,BL2017,BL2017b,BL2018,BLLW} 
	for the
	acoustic (medium) scattering to the electromagnetic medium as well as  source scattering.

The key result concerning the problem \eqref{eq:MaxwellCor}, which will be substantially used later to obtain results for the aforementioned related topics, is stated as follows. 
\begin{thm}\label{thm:main1}
	For any given $\mathbf{F}_j\in C^\alpha(\overline{\Cor^{r_0}})^3$, $j=1,2$, with $\alpha\in (0, 1)$, we consider the Maxwell problem \eqref{eq:MaxwellCor}. Suppose there exists a pair of solutions $\pare{\EM{E},\EM{H}}\in H(\mathrm{curl}, \Cor^{r_0})\times H(\mathrm{curl}, \Cor^{r_0})$ to \eqref{eq:MaxwellCor}. 
	Then there must hold that 
	\begin{equation}\label{eq:vanishing1}
	\mathbf{F}_1(x_0)=\mathbf{F}_2(x_0)=0. 
	\end{equation}
\end{thm}
\noindent We would like to emphasize that Theorem~\ref{thm:main1} and the problem \eqref{eq:MaxwellCor} concerns only the Maxwell equations locally in the truncated polyhedral cone $\Cor_{x_0}^{r_0}$. This allows for great generality for various applications in source or medium scattering as mentioned before. Roughly speaking, only local conditions of coefficients or fields around the corner need to be assumed, in order for results on scattering or inverse scattering  in latter sections to be valid.  Moreover, we allow corners with three or more than three number of edges.
	
	There are several technical developments in the current article. %, compare to the previous work concerning the Helmholz equation. 
	First, we provide a unified framework for several different topics in both medium and source scattering, via the local problem \eqref{eq:MaxwellCor} and Theorem~\ref{thm:main1}. It is in sharp difference with the existing studies %for the acoustic scattering 
	in the aforementioned literature, which usually deal with those topics individually and focus mainly on medium problems for acoustic scattering.
	%in studying the aforementioned different topics in wave scattering. The Maxwell problem \eqref{eq:MaxwellCor} and the corresponding vanishing property in Theorem~\ref{thm:main1} shall play a key role. 
%	In fact, we show that through certain reductions, %some of which are subtle and indirect, 
%	the study of those different topics in source or medium scattering can be reduced to the Maxwell problem \eqref{eq:MaxwellCor}. Then by using the vanishing property in Theorem~\ref{thm:main1}, %along with some further deductions, 
%	one can establish the desired results. 
	This unification reveals certain interesting %and mysterious 
	connections between those topics arising from different applications.  
	Second, in order to deal with the Maxwell system and derive the desired result, we need to develop new techniques to deal with the case of systems. In particular, some estimates and nonvanishing properties, which might be more natural for scalar field when dealing with the Helmholtz equation, become nontrivial in our case if following the same arguments as before. 
	Moreover, Theorem~\ref{thm:main1} and its proof  %in particular, vector fields, we need to d extending our earlier technical developments in \cite{Bsource,BL2016,BL2017,BL2017b,BL2018} in dealing with the Helmholtz system for the acoustic scattering, 
	provide a new, clearer and more elegant treatment %of the Maxwell system 
	for electromagnetic scattering. This gives more technical insights and also paves the way for further developments.		
	
	%Instead of introducing more details about our results for each of the topics mentioned above as consequences of Theorem~\ref{thm:main1}, we choose to provide more discussions after proving the main result. 
	
	The rest of the paper is organized as follows. In Section 2, we give the proof of Theorem~\ref{thm:main1}. In Section 3, we consider the electromagnetic scattering from active sources. We establish a geometric characterization
	of non-radiating sources; see, Theorem~\ref{thm:main2} and its corollaries. The unique recovery results of the supports of radiating sources areis included in Theorems~\ref{thm:main3} and \ref{thm:main4}. In Section 4, we consider the electromagnetic scattering from medium scatterers due to incident waves. We derive a geometric structure of electromagnetic interior transmission eigenfunctions in Theorem~\ref{thm:main5}. An implication to invisibility cloaking is given in Theorem~\ref{thm:main6} and that to unique determination of the supports of electromagnetic media isum in Theorem~\ref{thm:main7}.

	\section{Proof of Theorem~\ref{thm:main1}}
	
	In this section, we give the proof of Theorem~\ref{thm:main1}. To that end, we first present two auxiliary lemmas. 
	
	\begin{lem}\label{lem:Integral}
		Let $\Omega$ be a bounded Lipschitz domain in $\RR^3$ and $\EM{J}_j\in L^2(\Omega; \mathbb{C}^3)$.
		Suppose that $\pare{\EM{E},\EM{H}}\in H(\mathrm{curl},\Omega)\times H(\mathrm{curl},\Omega)$ is a solution to the Maxwell system 
		\begin{equation}\label{eq:MaxwellO}
			\nabla\wedge  \EM{E}-\im\omega \Mp_0 \EM{H}=\EM{J}_1,\quad \nabla\wedge \EM{H}+\im\omega\Ep_0 \EM{E}=\EM{J}_2\qquad \mbox{in $\Omega$}.
		\end{equation}
		Then one has
		\begin{equation}\label{eq:IntId1}
			\int_{\Omega} \EM{J}_1\cdot\EM{W}
			+\int_{\Omega} \EM{J}_2\cdot\EM{V}
			=\int_{\partial \Omega} \EM{W}\cdot\pare{\nor\cros\EM{E}}
			+\int_{\partial \Omega} \EM{V}\cdot\pare{\nor\cros\EM{H}},
		\end{equation}
		and
		\begin{equation}\label{eq:IntId2}
			\Ep_0\int_{\Omega} \EM{J}_1\cdot\EM{V}
			-\Mp_0\int_{\Omega} \EM{J}_2\cdot\EM{W}
			=\Ep_0\int_{\partial \Omega} \EM{V}\cdot\pare{\nor\cros\EM{E}}
			-\Mp_0\int_{\partial \Omega} \EM{W}\cdot\pare{\nor\cros\EM{H}},
		\end{equation}
		for any $\pare{\EM{V},\EM{W}}\in H(\mathrm{curl},\Omega)\times H(\mathrm{curl},\Omega)$ satisfying 
		\begin{equation}\label{eq:MaxwellTest}
			\nabla\wedge  \EM{V}-\im\omega \Mp_0 \EM{W}=0,\quad \nabla\wedge \EM{W}+\im\omega\Ep_0 \EM{V}=0\qquad \mbox{in $\Omega$}.
		\end{equation}
	\end{lem}
	\begin{proof}
		Applying $\EM{V}\in H(\mathrm{curl},\Omega)$ as a test function of the Maxwell system \eqref{eq:MaxwellO} and integrating by parts yields
		\begin{equation}\label{eq:a1}
			\begin{split}
				\int_{\Omega} \EM{J}_1\cdot\EM{V}
				&=\int_{\Omega} \pare{\nabla\wedge  \EM{E}-\im\omega \Mp_0 \EM{H}}\cdot \EM{V}
				\\&=\int_{\partial \Omega} \EM{V}\cdot\pare{\nor\cros\EM{E}}
				+\int_{\Omega}\im \omega \pare{\Mp_0 \EM{W}\cdot\EM{E}-\Mp_0\EM{H}\cdot\EM{V}}, 
			\end{split}	
		\end{equation}
		and
		\begin{equation}\label{eq:a4}
			\begin{split}
				\int_{\Omega} \EM{J}_2\cdot\EM{V}
				&=\int_{\Omega} \pare{\nabla\wedge \EM{H}+\im\omega\Ep_0 \EM{E}}\cdot \EM{V}
				\\&=\int_{\partial \Omega} \EM{V}\cdot\pare{\nor\cros\EM{H}}
				+\int_{\Omega}\im \omega \pare{ \Mp_0\EM{W}\cdot\EM{H}+ \Ep_0\EM{E}\cdot\EM{V}}, 
			\end{split}	
		\end{equation}
		where we have used the property that $\pare{\EM{V},\EM{W}}$ satisfies the Maxwell system \eqref{eq:MaxwellTest}.
		In a similar way, one can obtain
		\begin{equation}\label{eq:a3}
			\begin{split}
				\int_{\Omega} \EM{J}_1\cdot\EM{W}
				&=\int_{\partial \Omega} \EM{W}\cdot\pare{\nor\cros\EM{E}}
				-\int_{\Omega}\im \omega \pare{ \Ep_0\EM{V}\cdot\EM{E}+ \Mp_0\EM{H}\cdot\EM{W}}, 
			\end{split}	
		\end{equation}
		and
		\begin{equation}\label{eq:a2}
			\int_{\Omega} \EM{J}_2\cdot\EM{W}
			=\int_{\partial \Omega} \EM{W}\cdot\pare{\nor\cros\EM{H}}
			-\int_{\Omega}\im \omega  \pare{\Ep_0\EM{V}\cdot\EM{H}-\Ep_0\EM{E}\cdot\EM{W}}.
		\end{equation}
		Now, the identities \eqref{eq:IntId1} and \eqref{eq:IntId2} can be straightforwardly verified by using \eqref{eq:a1}--\eqref{eq:a2}. 
	\end{proof}
	
	\begin{lem}\label{lem:Nonvanish}
		Let $\Cor=\Cor_{w_1,\ldots,w_n;0}$ be a convex polyhedral cone with $n$ number of edges, $n\ge 3$. Given a constant $k$, a nontrivial constant complex vector $\EM{F}_0$, 
		there exist positive constants $c_{\Cor}$ and $C_{\Cor}$, and vectors $d,d^{\perp}\in\Ss$ with $d^\perp\perp d$, which
		satisfy that
		\begin{equation}\label{eq:innerdot}
			d \cdot \theta < - c_{\Cor} \quad \mbox{for any $\theta \in \Cor \cap \mathbb S^2$},
		\end{equation}
		and that
		\begin{equation}\label{eq:NonvanishIntegral}
			\abs{\int_{\Cor} e^{\rho\cdot x} dx} \geq C_{\Cor} \tau^{-3},\quad\mbox{ for any $\tau\geq k$},
		\end{equation}
		where the complex vector $\rho$ is given by
		\begin{equation}
			\rho := \tau d + \im \sqrt{\tau^2+k^2} d^\perp.
		\end{equation} 
		Moreover, the vectors $d$ and $d^\perp$ can be chosen in such a way that, denoting
		\begin{equation}\label{eq:p}
			p:=d^\perp- \im \sqrt{1+k^2/\tau^2} d,
		\end{equation}
		one has
		\begin{equation}\label{eq:relationVects}
			p\cdot\rho=0,\qquad \rho\cros p=-k^2 \pare{d\cros d^{\perp}}/\tau,
		\end{equation}
		and the limit below exists and satisfies
		\begin{equation}\label{eq:nontrivialFactor}
			\lim_{\tau\to\infty}\EM{F}_0\cdot p\neq 0.
		\end{equation}
	\end{lem}

	We postpone the proof of Lemma~\ref{lem:Nonvanish} to the end of this section and first present the proof of Theorem~\ref{thm:main1}. 
	
	\begin{proof}[Proof of Theorem~\ref{thm:main1}]
		We assume without loss of generality that $x_0=0$, which can be achieved by a rigid change of coordinates.	
		Recall from Lemma~\ref{lem:Integral} and the boundary condition in \eqref{eq:MaxwellCor}
		that, for any $\pare{\EM{V},\EM{W}}$ satisfying \eqref{eq:MaxwellTest}, there holds
		\begin{equation}\label{eq:IntIdProof}
			\int_{\Cor^{r_0}} \EM{F}_1\cdot\EM{W}
			+\int_{\Cor^{r_0}} \EM{F}_2\cdot\EM{V}
			=\int_{\partial \Cor^{r_0}\cap \partial B_{r_0}}
			\EM{W}\cdot\pare{\nor\cros\EM{E}}
			+\int_{\partial \Cor^{r_0}\cap \partial B_{r_0}} \EM{V}\cdot\pare{\nor\cros\EM{H}}.
		\end{equation}	
		We first prove $\EM{F}_2(0)=0$ by contradiction.
		
		Since $\mathbf{F}_j\in C^\alpha(\overline{\Cor^{r_0}})^3$, we can write
		\begin{equation*}
			\EM{F}_2=\EM{F}_0+\tilde{\EM{F}}
		\end{equation*}
		with $\EM{F}_0$ a constant vector, and $\tilde{\EM{F}}$ a vector field satisfying 
		\[
		|\tilde{\EM{F}}(x)|\le C|x|^{\alpha},\quad x\in \Cor^{r_0}.
		\]
		Under this splitting of $\EM{F}_2$, the equation \eqref{eq:IntIdProof} can be written as
		\begin{equation}\label{eq:TempProof}
			\begin{split}
				\int_{\Cor} \EM{F}_0\cdot \EM{V}
				=& 	\int_{\partial \Cor^{r_0}\cap \partial B_{r_0}} \EM{W}\cdot\pare{\nor\cros\EM{E}}
				+\int_{\partial \Cor^{r_0}\cap \partial B_{r_0}} \EM{V}\cdot\pare{\nor\cros\EM{H}}
				\\&+\int_{\Cor\setminus \Cor^{r_0}}\EM{F}_0\cdot \EM{V}
				-\int_{\Cor^{r_0}} \tilde{\EM{F}}\cdot\EM{V}
				-\int_{\Cor^{r_0}} \EM{F}_1\cdot\EM{W}.
			\end{split}
		\end{equation}
		We choose $\pare{\EM{V},\EM{W}}$, a pair of solutions to the Maxwell system \eqref{eq:MaxwellTest}, as
		\[
		\EM{V}(x)=pe^{\rho\cdot x}\quad\text{and}\quad 
		\EM{W}(x)=\frac{1}{\im \omega\mu_0} \rho\cros p e^{\rho\cdot x}
		\]
		with the complex vectors $\rho$ and $p$ given in Lemma~\ref{lem:Nonvanish} for $k^2=\omega^2\Ep_0\Mp_0$.
		
		We assume $\EM{F}_2(0)\neq 0$, namely, $\EM{F}_0\neq 0$.
		Concerning the LHS of \eqref{eq:TempProof}, we obtain from Lemma~\ref{lem:Nonvanish}, in particular from \eqref{eq:NonvanishIntegral} and \eqref{eq:nontrivialFactor}, that 
		\begin{equation}
			\abs{\int_{\Cor} \EM{F}_0 \cdot \EM{V}} = \abs{\EM{F}_0\cdot
				p} \abs{\int_{\Cor} e^{\rho\cdot x} dx} \ge
			\abs{\EM{F}_0\cdot p} C_{\Cor} \tau^{-3}\ge C_0\tau^{-3}>0
		\end{equation}
		holds for $\tau$ sufficiently large, with a constant $C_0$ (strictly) positive and independent of $\tau$.
		We shall show in the rest of the proof that the RHS of \eqref{eq:TempProof} is bounded by $C\tau^{-(3+\alpha)}$ and hence leads to a contradiction.
		
		We first deal with the terms in \eqref{eq:TempProof} concerning $\EM{V}$.
		For the integral over $\Cor^{r_0}$ we have
		\begin{equation*}
			\abs{\int_{\Cor^{r_0}} \tilde{\EM{F}}\cdot\EM{V}}
			\le \|\EM{F}_2\|_{C^\alpha}|p|\int_{\Cor^{r_0}} |x|^{\alpha}e^{\tau d\cdot x}dx
			\le 3\|\EM{F}_2\|_{C^\alpha}\tau^{-(3+\alpha)}\int_{\Cor}|y|^{\alpha}e^{d\cdot y}dy.
		\end{equation*}
		Recalling \eqref{eq:innerdot} from Lemma~\ref{lem:Nonvanish} one obtains
		\begin{equation*}
			\int_{\Cor}|y|^{\alpha}e^{d\cdot y}dy
			\le \int_{\Cor}|y|^{\alpha}e^{-c_{\Cor}|y|}dy
			\le 2\pi \int_{0}^{\infty}r^{2+\alpha}e^{-c_{\Cor}r}dr
			=C_{\Cor,\alpha}<\infty.
		\end{equation*}
		As a consequence,
		we have 
		\begin{equation}\label{eq:TempProof1}
			\abs{\int_{\Cor^{r_0}} \tilde{\EM{F}}\cdot\EM{V}}
			\le C_{\Cor,\alpha}\|\EM{F}_2\|_{C^\alpha}\tau^{-(3+\alpha)}.
		\end{equation}
		For the integral over $\Cor\setminus \Cor^{r_0}$ we can derive in a similar way that
		\begin{equation}
			\begin{split}
				\abs{\int_{\Cor\setminus \Cor^{r_0}}\EM{F}_0\cdot \EM{V}}
				\le& \abs{p\cdot \EM{F}_0}\int_{\Cor\setminus \Cor^{r_0}}e^{\tau d\cdot x}dx
				\le \abs{p\cdot \EM{F}_0}\int_{\Cor\setminus \Cor^{r_0}}e^{-c_{\Cor}\tau|x|}dx
				\\ \le& 2\pi\abs{p\cdot \EM{F}_0}\int_{r_0}^{\infty}r^{2}e^{-c_{\Cor}\tau r }dr
				\\=&2\pi\abs{p\cdot \EM{F}_0}e^{-c_{\Cor}\tau r_0}
				\int_{0}^{\infty}(r+r_0)^{2}e^{-c_{\Cor}\tau r}dr
				\\\le& C_{\Cor,r_0}\abs{p\cdot \EM{F}_0} e^{-c_{\Cor}\tau r_0}
				<\infty,
			\end{split}	
		\end{equation}
		when $\tau$ is
		sufficiently large.
		As for the boundary integral in \eqref{eq:TempProof} we have the estimate
		\begin{equation}\label{eq:TempProof2}
			\begin{split}
				\abs{\int_{\partial \Cor^{r_0}\cap \partial B_{r_0}} \EM{V}\cdot\pare{\nor\cros\EM{H}}}
				&\le 3\int_{\partial \Cor^{r_0}\cap \partial B_{r_0}} \abs{\nor\cros\EM{H}} e^{-c_{\Cor}\tau |x|}dx
				\\&\le 3C_{\Cor,r_0}e^{-c_{\Cor}r_0\tau}
				\|\EM{H}\|_{H(\curl,\, \Cor^{r_0})}.
			\end{split}	
		\end{equation}
		Recall the identity \eqref{eq:relationVects} from Lemma~\ref{lem:Nonvanish} that the modulus of $\EM{W}$ has the order of $\tau^{-1}$, with respect to $\tau$.
		Hence similar to \eqref{eq:TempProof2} we have
		\begin{equation}
			\abs{\int_{\partial \Cor^{r_0}\cap \partial B_{r_0}} \EM{W}\cdot\pare{\nor\cros\EM{E}}}
			\le 3C_{\Cor,r_0}k^2\tau^{-1}e^{-c_{\Cor}r_0\tau}
			\|\EM{E}\|_{H(\curl,\, \Cor^{r_0})}.
		\end{equation}
		Lastly, by using a similar argument as for deriving \eqref{eq:TempProof1}, one can obtain
		\begin{equation}
			\begin{split}
				\abs{\int_{\Cor^{r_0}} \EM{F}_1\cdot\EM{W}}
				\le \omega\Ep_0\tau^{-1}\|\EM{F}_1\|_{C^{0}}
				\int_{\Cor^{r_0}} e^{\tau d\cdot x}dx
				\le C_{\Cor}\omega\Ep_0\|\EM{F}_1\|_{C^{0}}\tau^{-4}.
			\end{split}
		\end{equation}
		
		In summary,
		the assumption $\EM{F}_2(0)\neq 0$ implies
		\begin{equation}
			C_0\tau^{-3}\le \tilde{C}_{\EM{F}_1,\EM{F}_2,\Cor,r_0,\alpha}\tau^{-(3+\alpha)}
		\end{equation}
		with $C_{0}>0$
		and this holds for any $\tau$ sufficiently large, which is impossible.
		Therefore, we have shown by contradiction that $\EM{F}_2(0)= 0$.
		Finally, one can verify $\EM{F}_1(0)= 0$ in the same way but by taking
		\[
		\EM{V}(x)=-\frac{1}{\im \omega\Ep_0} \rho\cros p e^{\rho\cdot x}\quad\text{and}\quad 
		\EM{W}(x)=pe^{\rho\cdot x}.
		\]
	\end{proof}

	\subsection{Proof of Lemma~\ref{lem:Nonvanish}}

	\begin{proof}
		We assume, up to some exchanges of notations,
		that the plane $\Span\{w_2,w_3\}$
		separates $w_1$ to the other side of space from $w_j$, $j=4,\ldots,n$.
		Notice that when the polyhedral cone contains only $n=3$ number of edges, this assumption holds automatically.
		Since the convex polyhedral cone $\Cor=\Cor_{w_1,\ldots,w_n;0}$ fits into a half-space, we can find a positive constant $\kappa$ and a vector $z\in\mathbb S^2$ satisfying 
		\begin{equation}
			z\cdot w_1=0\quad \text{and}\quad z\cdot
			w_j<-\kappa,\ j=2,\ldots,n.
		\end{equation}
		For any constant $s>0$, we define
		\begin{equation} \label{RepsDef}
			d=d_s := \frac{z-s w_1}{\abs{z-s w_1}}
			=\frac{z-s w_1}{\sqrt{1+s^2}}.
		\end{equation}
		and $d^{\perp}:=w_1\cros z$. It is noticed that $d^{\perp}$ is perpendicular to $d_s$, independent of the choice of $s$. We also set
		\begin{equation} \label{rhoDef}
			\rho=\rho_s := \tau d_s + \im\sqrt{\tau^2+k^2} d^{\perp}.
		\end{equation}
		
		Denote 
		\[
		\Cor_0 := \left\{ c_1 w_1 + c_2 w_2 + c_3 w_3;\ c_1,c_2,c_3 > 0
		\right\}.
		\]
		It is observed by straightforward calculation (see also, \cite[Proof of Lemma 3.4]{BL2017}) that 
		\begin{equation}\label{eq:intK0}
			\int_{\Cor_0} e^{\rho\cdot x} dx = \frac{\abs{w_1 \wedge w_2
					\wedge w_3} }{(-\rho\cdot w_1) (-\rho\cdot w_2)
				(-\rho\cdot w_3)},
		\end{equation}
		and that 
		\begin{equation}
			\abs{\frac{\tau}{\rho\cdot w_1}} 
			=\abs{\frac{1}{d_s\cdot w_1}}
			= \frac{\sqrt{1+s^2}}{s} >\frac{1}{s}.
		\end{equation}
		Moreover, we have for $\tau\ge k$ that
		\begin{equation}\label{eq:temp1}
			\begin{split}
				\abs{\frac{\tau}{\rho\cdot w_j}}& = 
				\frac{\sqrt{1+s^2}}{\abs{z\cdot w_j -
						s w_1\cdot w_j + \im\sqrt{1+k^2/\tau^2}
						\sqrt{1+s^2} d^{\perp}\cdot w_j}}
				\\&=\frac{\sqrt{1+s^2}}{\sqrt{ \abs{z\cdot w_j - s w_1\cdot w_j}^2 +
						\pare{1+k^2/\tau^2} \pare{1+s^2} \abs{d^{\perp}\cdot
							w_j}^2}} 
				\\&\ge \frac{\sqrt{1+s^2}}{\sqrt{ 2\pare{1+s^2}
						+	\pare{1+k^2/\tau^2} \pare{1+s^2} }} 
				\ge\frac{1}{2},\qquad j=2,3.
			\end{split}	
		\end{equation}
		Combining \eqref{eq:intK0}-\eqref{eq:temp1} yields
		\begin{equation} \label{K0integral}
			\abs{\tau^3 \int_{\Cor_0} e^{\rho\cdot x} dx} > \frac{\abs{ w_1
					\wedge w_2 \wedge w_3}}{4 s}.
		\end{equation}
		
		\medskip
		Let us prove an $s$-independent upper bound for the integral
		of $\exp(\rho\cdot x)$ over $\Cor\setminus \Cor_0$ next. This cone is
		generated by the vectors $w_2,\ldots,w_n$. Hence for any unit vector
		$\theta\in \Cor\setminus \Cor_0$ there are $\alpha_2,\ldots,\alpha_n \in
		\mathbb R_+ \cup \{0\}$ such that
		\[
		\theta = \alpha_2 w_2 + \ldots + \alpha_n w_n.
		\]
		Then
		\begin{equation} \label{ReRhoTheta}
			d_s\cdot\theta 
			= \tau \sum_{j=2}^n \alpha_j \frac{z\cdot w_j -
				s w_1\cdot w_j}{\sqrt{1+s^2}}.
		\end{equation}
		Recall that $z\cdot w_j < -\kappa$ when $j=2,\ldots,n$. 
		If 
		\begin{equation} \label{sUpperBound1}
			0<s \leq \kappa/3<1/3,
		\end{equation}
		then we have 
		\begin{equation}
			z\cdot w_j - s w_1\cdot	w_j \leq -2\kappa/3,\quad j=2,\ldots,n.
		\end{equation}
		As a consequence we obtain
		\begin{equation} \label{almost}
			d_s\cdot\theta \leq - \sum_{j=2}^n \alpha_j
			\frac{2\kappa}{3\sqrt{1+1/3^2}} \leq -\frac{1}{2}\kappa,
		\end{equation}
		where we have used the fact that
		\[
		1 = \abs{\theta} \leq \sum_{j=2}^n \alpha_j \abs{w_j} = \sum_{j=2}^n
		\alpha_j .
		\]
		Finally, this
		gives
		\begin{equation}\label{upperB}
			\begin{split}
				\abs{\tau^3 \int_{\Cor\setminus \Cor_0} e^{\rho\cdot x} dx} 
				&\leq
				\tau^3\int_{\pare{\Cor\setminus \Cor_0} \cap \mathbb S^2}  \int_0^\infty
				e^{r\tau d_s\cdot\hat{x} } r^2 dr d\hat{x} 
				\\ &
				\leq \abs{\pare{\Cor\setminus \Cor_0} \cap \mathbb S^2} \tau^3 \int_0^\infty
				e^{-\kappa\tau r/2} r^2 dr 
				\le 16\pi\kappa^{-3} .
			\end{split}
		\end{equation}
		
		Let us have $s$ so small that the right-hand side of
		\eqref{K0integral} is larger than the one in \eqref{upperB}, for
		example if
		\begin{equation}
			s<s_0:=\frac{\abs{w_1\wedge w_2\wedge w_3} }{128\pi} \kappa^3
		\end{equation}
		and $s_0\le\kappa/3$ to satisfy \eqref{sUpperBound1}.
		In this case, for any $s\in(0,s_0)$ we obtain
		\begin{equation}
			\begin{split}
				\abs{\int_{\Cor} e^{\rho\cdot x} dx} 
				&\geq \abs{\int_{\Cor_0} e^{\rho\cdot x}
					dx} - \abs{\int_{\Cor\setminus \Cor_0} e^{\rho\cdot x} dx} 
				\\&>\frac{\abs{ w_1	\wedge w_2 \wedge w_3}}{4 s}\tau^{-3}
				-16\pi\kappa^{-3}\tau^{-3}
				>16\pi\kappa^{-3}\tau^{-3}.
			\end{split}
		\end{equation}
		
		We are left to verify the property \eqref{eq:nontrivialFactor}.
		Let
		\[
		p_s:=d^{\perp}-\im\sqrt{1+k^2/\tau^2}d_s.
		\]
		We shall show that $\lim_{\tau\to\infty}\EM{F}_0\cdot p_s$ exists and can not vanish for all $s\in(0,s_0)$.
		We write $\EM{F}_0$ as
		\[
		\EM{F}_0=b_1w_1+b_2z+b_3d^{\perp}.
		\]
		Then
		\[
		\EM{F}_0\cdot p_s=b_3+\im\sqrt{1+k^2/\tau^2}\,\frac{sb_1-b_2}{\sqrt{1+s^2}} \longrightarrow b_3 + i \frac{sb_1-b_2}{\sqrt{1+s^2}},
		\]
		as $\tau\to\infty$. The right-hand side is a real-analytic function of $s$.
		If it would vanish in the open interval $(0,s_0)$, then it must be zero everywhere. Considering its values at $s=0$, $s=1$ and $s\to\infty$ yields $b_1=b_2=b_3=0$ and hence leads to the contradiction that $\EM{F}_0=0$.
		
		In conclusion, one can find $s\in(0,s_0)$, $d=d_s$ as in \eqref{RepsDef}, $d^{\perp}\in\Ss$ with $d^{\perp}\cdot d=0$ satisfying \eqref{eq:innerdot}, \eqref{eq:NonvanishIntegral} and \eqref{eq:nontrivialFactor}.
		Lastly, the equation \eqref{eq:relationVects} can be verified by straightforward computations.
		
	\end{proof}

	\section{Non-radiating sources and inverse source scattering problems}

	In this section, we are concerned with the electromagnetic scattering induced by an active source. The source is characterized by two vectorial functions $\mathbf{J}_1\in L^2(\Omega; \mathbb{C}^3)$ and $\mathbf{J}_2\in L^2(\Omega; \mathbb{C}^3)$, which are, respectively, referred to as the electric and the magnetic current densities. Here, %it is emphasized that 
	in order to appeal for a general mathematical study, we consider the possible presence of both electric and magnetic sources, though only the electric source might be the physically meaningful one. The source radiates electromagnetic waves and satisfies the following Maxwell system,
	\begin{equation}\label{eq:Maxwell1}
		\begin{split}
			&\nabla\wedge \EM{E}(x)-\im\omega \Mp_0 \EM{H}(x)=\EM{J}_1(x),\quad x\in\RR^3,\\
			& \nabla\wedge \EM{H}(x)+\im\omega\Ep_0 \EM{E}(x)=\EM{J}_2(x),\quad\, x\in\RR^3,\\
			& \lim_{|x|\to\infty} |x|\pare{\sqrt{\mu_0}\EM{H}\times\frac{x}{|x|}-\sqrt{\varepsilon_0}\EM{E}}=0,
		\end{split}
	\end{equation}
	where $\omega\in\mathbb{R}_+$ signifies the frequency of the wave. The last limit in \eqref{eq:Maxwell1} is known as the Silver-M\"uller radiation condition which holds uniformly in all directions
	$\hat x:=x/|x|\in\mathbb{S}^2$, $x\in\mathbb{R}^3\setminus\{0\}$, and characterizes the outgoing nature of the electromagnetic waves.
	As a consequence, one has the following asymptotics as $|{x}|\rightarrow+\infty$ (cf. \cite{CoK13}),
	\begin{equation}\label{eq:radiation}
	\pare{\EM{E},\EM{H}}(x)=\frac{e^{\im k |x|}}{|x|} \pare{\far{\EM{E}},\far{\EM{H}}}(\hat{x})+\mathcal{O}\left(\frac{1}{|x|^2}\right), 
	\end{equation}
	where $k:=\omega\sqrt{\varepsilon_0\mu_0}$ is known as the wavenumber. The two fields $\mathbf{E}_\infty(\hat x)$ and $\mathbf{H}_\infty(\hat x)$ are known as, respectively,
	the electric and the magnetic far-field patterns. By the Rellich theorem (cf. \cite{CoK13}), they encode all the information of the scattered wave fields $\mathbf{E}$ and $\mathbf{H}$ in the exterior of any Lipschitz domain that encloses the support of $\EM{J}_1$ and $\EM{J}_2$. Moreover, $\mathbf{E}_\infty$ and $\mathbf{H}_\infty$ are analytic functions on the unit sphere $\mathbb{S}^2$ and satisfy the following one-to-one correspondence,
	\[
	\mathbf{H}_\infty(\hat x)=\hat x\wedge \mathbf{E}_\infty(\hat x)
	\quad\text{and}\quad
	\mathbf{E}_\infty(\hat x)=-\hat x\wedge \mathbf{H}_\infty(\hat x),\quad\ \ \forall \hat x\in\mathbb{S}^2. 
	\]
	The Maxwell system \eqref{eq:Maxwell1} is well understood. We refer to \cite{LRX16,Ned01} for the existence of a unique pair of solutions
	$(\mathbf{E},\mathbf{H})\in H_{\mathrm{loc}}(\mathrm{curl}, \mathbb{R}^3)\times H_{\mathrm{loc}}(\mathrm{curl}, \mathbb{R}^3)$. 
	
	An important inverse scattering problem that arises in practical applications is to recover the unknown/inaccessible source from its associated far-field measurement. That is, 
	\begin{equation}\label{eq:ip1}
		\mathbf{E}_\infty(\hat x), \hat x\in\mathbb{S}^2\mapsto (\Omega; \mathbf{J}_1, \mathbf{J}_2). 
	\end{equation}
	Since $\mathbf{E}$ is real analytic on the unit sphere, we see that $\mathbb{S}^2$ can actually be replaced by any open subset of the unit sphere by virtue of analytic continuation. In the generic case, the dimensions of the measurement data $\mathbf{E}(\hat x)$ (associated with a fixed frequency $\omega\in\mathbb{R}_+$) and the unknown source $(\mathbf{J}_1, \mathbf{J}_2)$ in \eqref{eq:ip1} are, respectively, two and three. Here, by dimension we mean the number of free variables of the underlying quantity. Hence, it is impractical to ask for the unique recovery of the inverse problem \eqref{eq:ip1}, and a more practical inverse scattering problem could be posed as follows,
	\begin{equation}\label{eq:ip11}
		\mathbf{E}_\infty(\hat x),\ \hat x\in\mathbb{S}^2\mapsto \Omega. 
	\end{equation}
	That is, instead of seeking to completely recover the unknown source functions, one intends to recover the location and the shape of the support of the source.
	In determining $\Omega$, it suffices to recover $\partial\Omega$, and hence one can easily verify that for any fixed frequency $\omega\in\mathbb{R}_+$, the inverse scattering problem in \eqref{eq:ip11} is formally posed. Nevertheless, we would like to point out that
	the inverse problem \eqref{eq:ip1} is linear whereas the inverse problem \eqref{eq:ip11} is nonlinear.  
	
	Associated with the inverse scattering problem \eqref{eq:ip11}, we are mainly concerned with the following two fundamental issues:
	\begin{enumerate}
		\item For what kind of source there is no radiation, namely $\mathbf{E}_\infty\equiv 0$? In such a case, the source is invisible to exterior measurements since if $\mathbf{E}_\infty\equiv 0$, one
		actually has by the Rellich theorem that $\mathbf{E}=\mathbf{H}\equiv 0$ in $\mathbb{R}^3\setminus\overline{\Omega}$. This kind of source is referred to as non-radiating or radiationless in the literature. 
		
		\item If the source is not invisible, namely that it is detectable, can one really identify it by using the corresponding far-field observation? This is the identifiability and unique recovery issue. Mathematically, it can be stated as follows. Suppose that $(\Omega; \mathbf{J}_1, \mathbf{J}_2)$ and $(\Omega'; \mathbf{J}'_1, \mathbf{J}'_2)$ are two electromagnetic source configurations and $\mathbf{E}_\infty$ and $\mathbf{E}_\infty'$ are the associated far-field patterns respectively. Can one conclude that 
		\begin{equation}\label{eq:unique1}
			\mathbf{E}_\infty(\hat x)=\mathbf{E}_\infty'(\hat x), \hat x\in\mathbb{S}^2\quad\mbox{only if}\quad \Omega=\Omega' ?
		\end{equation}
	\end{enumerate}
	We mention in passing some related uniqueness results in \cite{Bsource,KS1,KS2,Ikehata} for a similar inverse problem \eqref{eq:ip11} posed for the acoustic scattering. We would also like to mention that for the linear inverse problem \eqref{eq:ip1}, but with the measurement data given by $\mathbf{E}_\infty(\hat x, \omega)$ for all $\hat x\in\mathbb{S}^2$ and $\omega\in\mathbb{R}_+$, there is a vast amount of literature devoted to it, both theoretically and computationally. It is of different nature from the focus of the current study which is mainly concerned with a single far-field pattern. So it would be impossible for us to give a comprehensive review of that interesting topic. 
	
	Let us first consider the geometric characterization of radiationless sources. The study of radiationless sources has a long and colorful history, which dates back to Sommerfeld's theory of extended rigid electron in 1904 \cite{Som1,Som2}.  Many physicists had theoretically predicated the existence of non-radiating sources and it was even postulated that non-radiating charge distributions might be used as
	models for elementary particles and might lead to a ``theory of nature'' \cite{Blei,Boh,Deva2,Ehr,Frie,Gam,Gbu,Goe,Hoen,Kim,Mare}. As an easy example, for any $\Psi_1$ and $\Psi_2$ being $\mathbb{C}^3$-valued smooth functions with compact supports in $\mathbb{R}^3$, if one sets
	\begin{equation}\label{eq:nr1}
		\mathbf{J}_1:=\nabla\wedge \Psi_1-\im\omega \Mp_0 \Psi_2,\quad \mathbf{J}_2:=\nabla\wedge \Psi_2+\im\omega\Ep_0 \Psi_1, 
	\end{equation}
	then $(\mathbf{J}_1, \mathbf{J}_2)$ is radiationless. Using Theorem~\ref{thm:main1} in the previous section, we can derive a %novel 
	geometric characterization of radiationless sources.
	To that end, the following definition of admissible sources shall be needed for our subsequent study. 
	
	\begin{defn}\label{def:corner}
		Given a source function $\EM{J}$, it is said to belong to the class $\claS$ if the following conditions are fulfilled:
		\begin{enumerate}
			\item	There exist a bounded Lipschitz domain $\Omega\subset\RR^3$ with a connected complement and a function $\Phi\in L_{loc}^2(\mathbb{R}^3; \CC^3)$ satisfying that $\EM{J} = \chi_{\overline{\Omega}} \Phi$.
			
			\item There is exist a polyhedral cone $\Cor=\Cor_{x_0}$ with the apex $x_0\in\partial\Omega$ such that 
			\[
			\Omega\cap B_{2r_0}(x_0)=\Cor_{x_0}\cap B_{2r_0}(x_0)\ \mbox{for some $r_0\in\mathbb{R}_+$},
			\] 
			and
			\[
			\Phi\in C^{\alpha}(B_{2r_0}(x_0))\ \mbox{for some $\alpha\in(0,1)$.}
			\]
			
			\item %$x_0$ connects to $\infty$ in $\mathbb{R}^3\setminus\overline{\Omega}$ in the sense that 
			There is a path in $\RR^3\setminus\overline\Omega$ joining $x_0$ to infinity.
			
		\end{enumerate}	
		In this case, we also say that {$(x_0; \Cor^{r_0}_{x_0})$ is a (generalized) 
			corner of $\EM{J}$, and $\EM{J}$ is $C^{\alpha}$ regular at the corner}.	
	\end{defn}
	\begin{rem}
		We use the expression ``generalized'' because a corner $(x_0; \Cor^{r_0}_{x_0})$ of a source function $\EM{J}$ might be degenerated in the sense that the case $\supp \EM{J}\cap B_{\varepsilon_0}(x_0) =\emptyset$ with a constant $\varepsilon_0>0$ is admitted in Definition~\ref{def:corner}. 
		As a simplest example, the trivial source function $\EM{J}\equiv 0$ belongs to the class $\mathscr{A}$ and any point $x_0\in \RR^3$ is a (generalized) corner of $\EM{J}$. 
	\end{rem}
	
	Then we have
	
	\begin{thm}\label{thm:main2}
		Consider an electric source $\mathbf{J}_1$ and a magnetic source $\mathbf{J}_2$ that are both supported in $\Omega$. Suppose that both $\mathbf{J}_1$ and $\mathbf{J}_2$ belong to the class $\mathscr{A}$ and let $(x_0, \mathcal{K}_{x_0}^{r_0})$ be a (generalized) corner of $\mathbf{J}_1$ and $\mathbf{J}_2$. If $(\Omega;\mathbf{J}_1, \mathbf{J}_2)$ is radiationless, namely the far-field pattern $(\mathbf{E}_\infty, \mathbf{H}_\infty)$ of the Maxwell system \eqref{eq:Maxwell1} associated with the source $(\Omega; \mathbf{J}_1, \mathbf{J}_2)$ is identically zero, 
		then one must have that
		\begin{equation}\label{eq:cond2}
			\mathbf{J}_1(x_0)=\mathbf{J}_2(x_0)=0. 
		\end{equation}
	\end{thm}
	
	Before presenting the proof of Theorem~\ref{thm:main2}, we first present some interesting
	consequences. 
	The first
	one is a geometric characterization of a function space. 
	
	\begin{cor}
		Let $L_c^2(\mathbb{R}^3; \mathbb{C}^3)$ denote the space of $L_{loc}^2(\mathbb{R}^3; \mathbb{C}^3)$ functions with compact supports. Introduce the following function space,
		\begin{equation}\label{eq:fs1}
			\mathscr{F}:=\{\mathbf{F}\in L_c^2(\mathbb{R}^3; \mathbb{C}^3);\ \mathbf{F}=\nabla\wedge\nabla\wedge\mathbf{M}+c_0\mathbf{M},\ \ \mathbf{M}\in H_{loc}^2(\mathrm{curl}, \mathbb{R}^3) \},
		\end{equation}
		where $c_0$ is a nonzero constant. Then for any $\mathbf{F}\in\mathscr{F}\cap \mathscr{A}$, one has that $\mathbf{F}$ must be vanishing at its corner points. 
	\end{cor}
	\begin{proof}
		We first consider the case that $c_0\in\mathbb{R}_+$. Suppose that $\mathbf{F}\in\mathscr{F}\cap \mathscr{A}$. Set $\omega=\sqrt{c_0}$ and $\varepsilon_0=\mu_0=1$, and
		\[
		\mathbf{J}_1=-\frac{1}{\mathrm{i}\omega\varepsilon_0}\mathbf{F},\quad \mathbf{J}_2\equiv 0. 
		\] 
		By straightforward calculations, one can verify that $\mathbf{J}_1$ is a radiationless magnetic source and hence by Theorem~\ref{thm:main2}, it must be vanishing at its corner points. Similarly, for the case $c_0\in\mathbb{R}_-$, by setting
		\[
		\mathbf{J}_1\equiv 0,\quad \mathbf{J}_2=\frac{1}{\mathrm{i}\omega\mu_0}\mathbf{F},
		\] 
		one can directly verify that $\mathbf{J}_2$ is a radiationless electric source and hence it must be vanishing at its corner points. 
	\end{proof}
	
	\begin{cor}\label{cor:3}
		Consider an electromagnetic source $(\Omega; \mathbf{J}_1, \mathbf{J}_2)$ as described in Theorem~\ref{thm:main2}. 
		For each fixed $j\in\{1,2\}$, if $(x_0;\mathcal{K}_{x_0}^{r_0})$ is a corner of $\mathbf{J}_j$ in the sense of Definition~\ref{def:corner}, then the source must radiate a nonzero far-field pattern whenever $\mathbf{J}_j(x_0)\neq 0$.
	\end{cor}
	
	Theorem~\ref{thm:main2} and Corollary~\ref{cor:3} 
 give geometric characterization and classification of radiating and radiationless electromagnetic sources. We next present the proof of Theorem~\ref{thm:main2}. 
	
	\begin{proof}[Proof of Theorem~\ref{thm:main2}]
		Suppose that $(\Omega;\mathbf{J}_1, \mathbf{J}_2)$ is radiationless, namely $(\mathbf{E}_\infty, \mathbf{H}_\infty)\equiv 0$. Then by the Rellich theorem
		(cf. \cite{CoK13}), we know that $\mathbf{E}=\mathbf{H}=0$ in the component of $\mathbb{R}^3\setminus\overline{\Omega}$ that is unbounded. Hence, we see from \eqref{eq:Maxwell1} and the definition of $\claS$ (Definition~\ref{def:corner}) that
		\begin{equation}\label{eq:Maxwell12}
			\begin{cases}
				&\nabla\wedge \EM{E}(x)-\im\omega \Mp_0 \EM{H}(x)=\EM{J}_1(x),\quad\, x\in \mathcal{K}_{x_0}^{r_0} ,\\
				& \nabla\wedge \EM{H}(x)+\im\omega\Ep_0 \EM{E}(x)=\EM{J}_2(x),\quad\ x\in \mathcal{K}_{x_0}^{r_0} ,\\
				&\nu(x)\wedge\mathbf{E}(x)=\nu(x)\wedge\mathbf{H}(x)=0,\quad x\in\partial \mathcal{K}_{x_0}^{r_0}\setminus\partial B_{r_0}(x_0),
			\end{cases}
		\end{equation}
		for any (external) corner point $x_0$ of $\Omega$.
		Therefore, we readily have from Theorem~\ref{thm:main1} that
		\[
		\mathbf{J}_1(x_0)=\mathbf{J}_2(x_0)=0.
		\]
	\end{proof}
	
	We proceed to deal with the unique recovery issue of the nonlinear inverse scattering problem \eqref{eq:ip11}. We first present a local unique recovery result as follows. 
	\begin{thm}\label{thm:main3}
		Let $(\Omega; \mathbf{J}_1, \mathbf{J}_2)$ and $(\Omega'; \mathbf{J}'_1, \mathbf{J}'_2)$ be two electromagnetic source configurations where $\mathbf{J}_1,\mathbf{J}_2,\mathbf{J}'_1,\mathbf{J}'_2$ belong to the class $\mathscr{A}$ in 
		terms of Definition~\ref{def:corner} such that $\mathbf{J}_1,\mathbf{J}_2$ are supported on $\Omega$ and $\mathbf{J}'_1,\mathbf{J}'_2$ on $\Omega'$. Let $\mathbf{E}_\infty$ and $\mathbf{E}_\infty'$ be the electric far-field patterns associated with $(\Omega; \mathbf{J}_1, \mathbf{J}_2)$ and $(\Omega'; \mathbf{J}'_1, \mathbf{J}'_2)$, respectively. If $\mathbf{E}_\infty(\hat x)=\mathbf{E}_\infty'(\hat x)$ for all $\hat x\in\mathbb{S}^2$. Then the set difference
		\begin{equation}\label{eq:diff1}
			\Omega\Delta\Omega':=(\Omega\setminus\overline{\Omega'})\cup (\Omega'\setminus\overline{\Omega}), 
		\end{equation}
		cannot contain a corner whose apex, say $x_0$, connects to infinity in the unbounded component of $\mathbb{R}^3\setminus\overline{\Omega\cup\Omega'}$ and satisfies
		\begin{equation}\label{eq:cc11}
			\begin{split}
				(i)&.~\mbox{$\mathbf{J}_1(x_0)\neq 0$ or $\mathbf{J}_2(x_0)\neq 0$ if $x_0$ is a corner of $\Omega$}, \\ (ii)&.~\mbox{$\mathbf{J}'_1(x_0)\neq 0$ or $\mathbf{J}'_2(x_0)\neq 0$ if $x_0$ is a corner of $\Omega'$}.
			\end{split}
		\end{equation}
	\end{thm}
	
	\begin{proof}
		We prove the theorem by a reductio ad absurdum. Let $(\mathbf{E}, \mathbf{H})$ and $(\mathbf{E}',\mathbf{H}')$ be the electromagnetic fields of the Maxwell system \eqref{eq:Maxwell1} associated with $(\Omega; \mathbf{J}_1, \mathbf{J}_2)$ and $(\Omega'; \mathbf{J}'_1, \mathbf{J}'_2)$, respectively. Set $\mathbf{G}$ to denote the unbounded connected component of $\mathbb{R}^3\setminus\overline{\Omega\cup\Omega'}$. By the Rellich theorem
		and the fact that $\mathbf{E}_\infty=\mathbf{E}'_\infty$, we readily have that
		\begin{equation}\label{eq:ar1}
			(\mathbf{E}, \mathbf{H})=(\mathbf{E}', \mathbf{H}')\quad\mbox{in}\ \ \mathbf{G}. 
		\end{equation}
		
		Without loss of generality, we assume that $x_0\in\overline{\Omega}\setminus\Omega'$, and we let $\mathcal{K}_{x_0}^{r_0}$ with $r_0\in\mathbb{R}_+$ sufficiently small, be a corner of $\Omega$ such that $\mathcal{K}_{x_0}^{r_0}\subset\overline{\Omega}\setminus\Omega'$. Clearly, there hold
		\begin{equation}\label{eq:ar2}
			\begin{cases}
				&\nabla\wedge \EM{E}(x)-\im\omega \Mp_0 \EM{H}(x)=\EM{J}_1(x),\quad\, x\in\mathcal{K}_{x_0}^{r_0},\\
				& \nabla\wedge \EM{H}(x)+\im\omega\Ep_0 \EM{E}(x)=\EM{J}_2(x),\quad\ x\in\mathcal{K}_{x_0}^{r_0},
			\end{cases}
		\end{equation}
		and 
		\begin{equation}\label{eq:ar3}
			\begin{cases}
				&\nabla\wedge \EM{E}'(x)-\im\omega \Mp_0 \EM{H}'(x)=0,\quad\, x\in\mathcal{K}_{x_0}^{r_0},\\
				& \nabla\wedge \EM{H}'(x)+\im\omega\Ep_0 \EM{E}'(x)=0,\quad\ x\in\mathcal{K}_{x_0}^{r_0},
			\end{cases}
		\end{equation}
		Set
		\[
		\widetilde{\EM{E}}=\mathbf{E}-\mathbf{E}',\quad\widetilde{\EM{E}}=\mathbf{H}-\mathbf{H}'.
		\]
		By \eqref{eq:ar1}--\eqref{eq:ar3}, one sees that 
		\begin{equation}\label{eq:ar4}
			\begin{cases}
				&\nabla\wedge \widetilde{\EM{E}}(x)-\im\omega \Mp_0 \widetilde{\EM{H}}(x)=\EM{J}_1(x),\quad\, x\in\mathcal{K}_{x_0}^{r_0},\\
				& \nabla\wedge \widetilde{\EM{H}}(x)+\im\omega\Ep_0 \widetilde{\EM{E}}(x)=\EM{J}_2(x),\quad\ x\in\mathcal{K}_{x_0}^{r_0},\\
				&\nu(x)\wedge\widetilde{\mathbf{E}}(x)=\nu(x)\wedge\widetilde{\mathbf{H}}(x)=0,\quad x\in\partial\mathcal{K}_{x_0}^{r_0}\cap\partial\mathbf{G}. 
			\end{cases}
		\end{equation}
		Hence, by Theorem~\ref{thm:main1}, we readily have that
		\[
		\mathbf{J}_1(x_0)=\mathbf{J}_2(x_0)=0,
		\]
		which is a contradiction to \eqref{eq:cc11}. 
	\end{proof}
	
	As an immediate consequence of Theorem~\ref{thm:main3}, we have
	
	\begin{thm}\label{thm:main4}
		Let $(\Omega; \mathbf{J}_1, \mathbf{J}_2)$ be an electromagnetic source from the class $\mathscr{A}$ in 
		terms of Definition~\ref{def:corner}. Suppose that $\Omega$ is a convex polyhedron and at each of its apexes, either $\mathbf{J}_1$
		or $\mathbf{J}_2$ is non-vanishing. Then $\mathbf{E}_\infty$ uniquely determines $\Omega$.  
	\end{thm}
	\begin{proof}
		Suppose that there exists another electromagnetic source $(\Omega'; \mathbf{J}'_1, \mathbf{J}'_2)$ with $\mathbf{E}_\infty=\mathbf{E}_\infty'$. If $\Omega\neq\Omega'$ then there is a corner $x_0$ of let's say $\Omega$ such that $x_0 \in \overline{\Omega}\setminus\Omega'$ and it can be connected to infinity outside of $\Omega\cup\Omega'$. But since $\EM{J}_1(x_0)\neq0$ or $\EM{J}_2(x_0)\neq0$ Theorem~\ref{thm:main3} implies that $x_0$ cannot be a corner of $\Omega\Delta\Omega'$. In other words $x_0 \in \overline{\Omega'}$ too. The contradiction gives $\Omega=\Omega'$. 
	\end{proof}
	
	It is remarked that one can show the same unique recovery result
	as that in Theorem~\ref{thm:main4} for a bit more general case where the source support $\Omega$ consists of finitely many disjoint convex polyhedra. 
	
	\section{Inverse medium scattering and interior transmission eigenvalue problem}
	
	In this section, we consider another scenario of practical importance where the electromagnetic scattering is induced by an inhomogeneous medium and an incident wave field. Suppose an inhomogeneous medium is embedded in a
	uniformly homogeneous space with electric permittivity $\varepsilon_0$ and magnetic permeability $\mu_0$. The inhomogeneous medium is supported in $\Omega$ and is characterized by its material parameters including the electric permittivity $\varepsilon\in L^\infty(\Omega; \mathbb{R}_+)$, magnetic permeability $\mu\in L^\infty(\Omega; \mathbb{R}_+)$ and electric conductivity $\sigma\in L^\infty(\Omega; \mathbb{R}_+^0)$. Throughout the rest of the paper, it is assumed that $\Omega$ is a bounded Lipschitz domain in $\mathbb{R}^3$ with a connected complement $\mathbb{R}^3\setminus\overline{\Omega}$. In what follows, for notational convenience, we extend $\varepsilon, \mu$ and $\sigma$ to the whole space $\mathbb{R}^3$ by setting $\varepsilon(x)=\varepsilon_0$, $\mu(x)=\mu_0$ and $\sigma(x)=0$ for $x\in\mathbb{R}^3\setminus\overline{\Omega}$. Associated with the scattering medium $(\Omega; \varepsilon, \mu, \sigma)$ described above, the electromagnetic scattering is then induced by sending a wave field $(\mathbf{E}^i, \mathbf{H}^i)$ impinging on $\Omega$. It is a pair of entire solutions to the following homogeneous Maxwell system
	\begin{equation}\label{eq:Maxwellh}
		\curl  \EM{E}^i-\im\omega \Mp_0 \EM{H}^i=0,\quad \curl \EM{H}^i+\im\omega\Ep_0 \EM{E}^i=0\quad\mbox{in}\ \ \mathbb{R}^3.
	\end{equation}
	The interaction of the incident field $(\mathbf{E}^i, \mathbf{H}^i)$ and the scattering body $(\Omega; \varepsilon, \mu, \sigma)$ generates electromagnetic wave scattering. We let $(\mathbf{E}, \mathbf{H})$ and $(\mathbf{E}^t, \mathbf{H}^t)$ denote, respectively, the scattered and the
	total electromagnetic fields. There hold 
	\[
	(\mathbf{E}^t, \mathbf{H}^t)=(\mathbf{E}^i, \mathbf{H}^i)+(\mathbf{E}, \mathbf{H})\quad\mbox{in}\ \ \mathbb{R}^3,
	\]
	and the following Maxwell system 
	\begin{equation}\label{eq:Maxwell2}
		\begin{split}
			& \nabla\wedge\mathbf{E}^t({x})-\mathrm{i}\omega\mu({x}) \mathbf{H}^t({x})=0,\hspace*{1.8cm} {x}\in\mathbb{R}^3,\medskip\\
			&\nabla\wedge\mathbf{H}^t({x})+(\mathrm{i}\omega\varepsilon({x})-\sigma({x})) \mathbf{E}^t({x})=0,\quad {x}\in\mathbb{R}^3,\medskip\\
			& \lim_{|x|\to\infty} |x|\pare{\sqrt{\mu_0}\EM{H}\times\frac{x}{|x|}-\sqrt{\varepsilon_0}\EM{E}}=0.
		\end{split}
	\end{equation}
	We refer to \cite{CoK13,LRX16,Ned01} for the 
	existence of a unique pair of solutions
	$(\mathbf{E},\mathbf{H})\in H_{\mathrm{loc}}(\mathrm{curl}, \mathbb{R}^3)\times H_{\mathrm{loc}}(\mathrm{curl}, \mathbb{R}^3)$ and the following far-field expansion 
	\begin{equation}\label{eq:farfield}
		\pare{\EM{E},\EM{H}}(x)=\frac{e^{\im k |x|}}{|x|} \pare{\far{\EM{E}},\far{\EM{H}}}(\hat{x})+\mathcal{O}\left(\frac{1}{|x|^2}\right), 
	\end{equation}
	where $k=\omega\sqrt{\varepsilon_0\mu_0}$.  
	
	Similar to the inverse source scattering problem \eqref{eq:ip1}, the inverse medium scattering problem can be stated as follows,
	\begin{equation}\label{eq:ip2}
		\mathbf{E}_\infty(\hat x),\ \hat x\in\mathbb{S}^2\mapsto (\Omega; \varepsilon,\mu, \sigma). 
	\end{equation}
	It can be verified directly that the inverse medium scattering problem \eqref{eq:ip2} is nonlinear and under-determined in the generic case. In what follows, similar to the inverse source 
	scattering case, we first consider the invisibility issue for the inverse medium scattering problem \eqref{eq:ip2}, namely $\mathbf{E}_\infty\equiv 0$.

	If $\mathbf{E}_\infty\equiv 0$, by the Rellich theorem,
	one has $(\mathbf{E}, \mathbf{H})=0$ in $\mathbb{R}^3\setminus\overline{\Omega}$. Hence, it is straightforward to show that in this case, there holds
	\begin{equation}\label{eq:ITP}
		\begin{cases}
			\nabla\wedge \EM{E}^t-\mathrm{i}\omega\Mp \EM{H}^t=0,\quad
			\nabla \wedge \EM{H}^t+\mathrm{i}\omega\Eci \EM{E}^t=0
			&\mbox{in\ \ $\Omega$},\\
			\nabla\wedge \EM{E}^0-\mathrm{i}\omega\Mp_0 \EM{H}^0=0,\quad
			\nabla\wedge \EM{H}^0+\mathrm{i}\omega\Ep_0 \EM{E}^0=0
			&\mbox{in\ \ $\Omega$},\\
			\quad\nu\cros \EM{E}^t=\nu\cros \EM{E}^0,\quad
			\qquad\nu\cros \EM{H}^t=\nu\cros \EM{H}^0
			&\mbox{on\ \ $\partial \Omega$},
		\end{cases}
	\end{equation}
	where $\gamma:=\varepsilon+\mathrm{i}\sigma/\omega$ with $(\Omega; \varepsilon,\mu,\sigma)$ introduced earlier. Equations \eqref{eq:ITP} is known as the interior transmission eigenvalue problem in the literature. 
	If for a certain $\omega\in\mathbb{R}_+$, there exist nontrivial $(\mathbf{E}^t,\mathbf{H}^t)\in H_{\mathrm{loc}}(\mathrm{curl}, \mathbb{R}^3)\times H_{\mathrm{loc}}(\mathrm{curl}, \mathbb{R}^3)$ and $(\mathbf{E}^0,\mathbf{H}^0)\in H_{\mathrm{loc}}(\mathrm{curl}, \mathbb{R}^3)\times H_{\mathrm{loc}}(\mathrm{curl}, \mathbb{R}^3)$ satisfying \eqref{eq:ITP}, then $\omega$ is called an interior transmission eigenvalue and $(\mathbf{E}^t, \mathbf{H}^t), (\mathbf{E}^0, \mathbf{H}^0)$ are called the corresponding eigenfunctions. The interior transmission eigenvalue problem is an important type of non-self-adjoint problem in the spectral theory associated with wave phenomena and its study has a long and colorful history; see \cite{CGH,CHreview,CoK13,CM,CKP, RS,PS,Kir} and the references therein. From our discussion above, it is seen that if invisibility occurs, then $\omega$ is an interior transmission eigenvalue and the restrictions of the total wave field $(\mathbf{E}^t,\mathbf{H}^t)$ and incident wave field $(\mathbf{E}^0,\mathbf{H}^0)$ form the corresponding eigenfunctions. On the other hand, it is straightforward to show that if $(\mathbf{E}^0, \mathbf{H}^0)$ is an eigenfunction associated with $(\Omega; \varepsilon,\mu,\sigma)$, and can be
	extended to the whole space $\mathbb{R}^3$ to form a pair of entire solutions to the Maxwell system \eqref{eq:Maxwellh}, then as the incident field to the scattering system \eqref{eq:Maxwell2}, the resulting far-field pattern is identically zero; that is, invisibility occurs. In order to gain more insights about the invisibility, we first provide a geometric characterization of the interior transmission eigenfunctions. 
	
	\begin{thm}\label{thm:main5}
		Consider the interior transmission eigenvalue problem \eqref{eq:ITP}, and suppose that $(\mathbf{E}^t, \mathbf{H}^t)$ and $(\mathbf{E}^0, \mathbf{H}^0)$ are a pair of eigenfunctions associated with the eigenvalue $\omega\in\mathbb{R}_+$. Assume that $\Omega$ possesses a corner $\mathcal{K}_{x_0}^{r_0}$ and moreover,
		\begin{equation}\label{eq:rega1}
			(\mu-\mu_0)\mathbf{H}^t,\ (\gamma-\varepsilon_0)\mathbf{E}^t\in C^\alpha(\overline{\mathcal{K}_{x_0}^{r_0}})^3,
		\end{equation} 
		for some $\alpha\in (0, 1)$. Then there holds
		\begin{equation}\label{eq:rega2}
			(\mu-\mu_0)(x_0)\mathbf{H}^t(x_0)=(\gamma-\varepsilon_0)(x_0)\mathbf{E}^t(x_0)=0. 
		\end{equation}
	\end{thm}
	
	\begin{proof}
		
		By straightforward calculations, one can show by virtue of \eqref{eq:rega1} that $(\widehat{\mathbf{E}},\widehat{\mathbf{H}}):=(\mathbf{E}^t,\mathbf{H}^t)-(\mathbf{E}^0,\mathbf{H}^0)$ satisfies 
		\begin{equation}\label{eq:rega3}
			\left\{
			\begin{split}
				& \nabla\wedge  \widehat{\EM{E}}-\im\omega \Mp_0 \widehat{\EM{H}}=\EM{F}_{1} \quad \mbox{in $\Cor_{x_0}^{r_0}$},\\
				& \, \nabla\wedge \widehat{\EM{H}}+\im\omega\Ep_0 \widehat{\EM{E}}=\EM{F}_{2} \quad \mbox{in $\Cor_{x_0}^{r_0}$},\\
				&\, \nor\cros \widehat{\EM{E}}=\nor\cros\widehat{\EM{H}}=0 \hspace*{0.8cm}\mbox{on $\partial \Cor_{x_0}^{r_0}\cap\partial\Omega$ }, 
			\end{split}\right.
		\end{equation}
		where
		\begin{equation}\label{eq:rega4}
			\mathbf{F}_1=(\mu-\mu_0)\mathbf{H}^t\quad\mbox{and}\quad\mathbf{F}_2=(\gamma-\varepsilon_0)\mathbf{E}^t. 
		\end{equation}
		Hence, by Theorem~\ref{thm:main1} and \eqref{eq:rega1}, one readily has \eqref{eq:rega2}. 
	\end{proof}
	
	The study of the geometric structures of transmission eigenfunctions was initiated in \cite{BL2017b} %by two of the authors of the present article. In subsequent articles, more intriguing geometric structures of transmission eigenfunctions were discovered recently 
	and then further developed in \cite{Bsource,BLLW,DCL}. However, in all of the aforementioned literature, the transmission eigenvalue problems are associated to the Helmholtz system that arises 
	from the time-harmonic acoustic scattering. The intrinsic geometric structure of the interior transmission eigenfunctions associated with the Maxwell system in Theorem~\ref{thm:main5} is the first one of its type in the literature. By assuming that $(\mu-\mu_0)(x_0)\neq 0$ and $(\gamma-\varepsilon_0)(x_0)\neq 0$, one readily has from \eqref{eq:rega2} that
	\begin{equation}\label{eq:rega5}
		\mathbf{E}^t(x_0)=\mathbf{H}^t(x_0)=0. 
	\end{equation}
	This vanishing property at the corner point is consistent with most of the existing results for the interior transmission eigenfunctions associated with the acoustic scattering. 
	
	We would like to make two remarks regarding the regularity assumption \eqref{eq:rega1}. First, it would be interesting to investigate that, under what conditions of the medium configuration $(\Omega;\gamma,\mu,\varepsilon_0,\mu_0)$ and the interior transmission eigenvalue $\omega\in\mathbb{R}_+$, the corresponding transmission eigenfunctions shall fulfill the regularity condition \eqref{eq:rega1}. Second, we firmly believe that the regularity condition \eqref{eq:rega1} is a technical limitation, and the vanishing property \eqref{eq:rega2} should hold in a much more general scenario. To overcome this issue, one should try to relax the regularity condition in Theorem~\ref{thm:main1} for \eqref{eq:MaxwellCor}. However, the relaxation is fraught with challenges and we choose to leave it for future study. Another promising way to address the above two issues is to conduct the numerical investigation which we shall report in a forthcoming paper. 
	
	Now, we are in a position to consider the practical implication of the geometric property in Theorem~\ref{thm:main5} to invisibility in wave scattering. We have
	
	\begin{thm}\label{thm:main6}
		Consider the electromagnetic scattering problem \eqref{eq:Maxwell2}. If the medium $(\Omega;\varepsilon,\mu,\sigma)$ possesses a corner $\mathcal{K}_{x_0}^{r_0}$ in its support and moreover, 
		\begin{equation}\label{eq:rega11}
			(\mu-\mu_0)\mathbf{H}^t,\ (\gamma-\varepsilon_0)\mathbf{E}^t\in C^\alpha(\overline{\mathcal{K}_{x_0}^{r_0}})^3,
		\end{equation} 
		for some $\alpha\in (0, 1)$ and
		\begin{equation}\label{eq:rega21}
			(\mu-\mu_0)(x_0)\mathbf{H}_0(x_0)\neq
			0\quad\mbox{or}\quad (\gamma-\varepsilon_0)(x_0)\mathbf{E}_0(x_0) \neq
			0,
		\end{equation}
		then the corresponding far-field pattern cannot be identically vanishing; that is, invisibility does not occur. 
	\end{thm}
	
	\begin{proof}
		This is a direct consequence of the fact that if the far-field pattern is identically vanishing, then one has both a) scattered wave fields $\EM{H},\EM{E}$ vanishing at the boundary, and b) the interior transmission eigenvalue problem \eqref{eq:ITP}. Hence one arrives at a contradiction by using the vanishing property in Theorem~\ref{thm:main5} and the non-vanishing condition in \eqref{eq:rega21}. 
	\end{proof}
	
	Theorem~\ref{thm:main6} essentially indicates that if the underlying scattering medium possesses a corner, then it radiates a nonzero scattering pattern unless the incident and hence total wave fields vanish at the corner. We would like to emphasize the local nature of such a non-invisibility result. That is, the assertion of non-invisibility mainly comes from the ``strong'' radiating nature of the corner which is independent of the other parts of the scatterer. This is also in consistence with the corresponding studies in the literature for the acoustic case \cite{BL2016,BPS,BV,DCL,PSV}. However in Maxwell scattering one does not have $H^2$- or $C^\alpha$-smoothness a-priori. Hence we would like to point out that the technical condition \eqref{eq:rega11} again restricts the more practical applicability of our result. Similar to our earlier remarks made after Theorem~\ref{thm:main5}, in order to overcome this issue, one should consider relaxing the regularity assumption in Theorem~\ref{thm:main1} for \eqref{eq:MaxwellCor}. Theorem~\ref{thm:main6} points out a promising direction for further investigation. We believe that the regularity condition \eqref{eq:rega11} should be relaxed to a much more general scenario. Nevertheless, we would also like to point out that in a recent paper \cite{LX}, %by two of the authors of the current article, 
	by following a different pathway, the result on corner always scattering was also proved under a very mild condition imposed on the incident wave field. But in \cite{LX}, the corner should be of degree $90^\circ$, whereas in Theorem~\ref{thm:main6}, the corner could a generic one as long as it is not degenerate to be $180^\circ$.
	
	Finally, we mention in passing about invisibility cloaking, which is a topic that has received significant attentions in the last decade, and is related to our discussion above. Our results says that cloaking devices cannot have corners. This is a huge topic and we choose not to give more discussions and only refer to the survey papers \cite{GKLU4,GKLU5,LiuUhl} and the references therein for more relevant studies in that direction.   
	
	Theorem~\ref{thm:main6} shows that if a medium scatterer possesses a corner, then it is detectable. Next, we show that the detectability also implies the identifiability, namely, the unique recovery of the inverse scattering problem \eqref{eq:ip2}.
	In fact, we have 
	\begin{thm}\label{thm:main7}
		Let $(\Omega; \gamma, \mu)$ and $(\Omega'; \gamma', \mu)$ be two medium scatterers and, $(\mathbf{E}_t, \mathbf{H}_t, \mathbf{E}_\infty)$ and $(\mathbf{E}'_t, \mathbf{H}_t', \mathbf{E}'_\infty)$ be the associated total and far fields. If $\mathbf{E}_\infty(\hat x)=\mathbf{E}_\infty'(\hat x)$ for all $\hat x\in\mathbb{S}^2$, then the set difference
		$\Omega\Delta\Omega'$ as defined in \eqref{eq:diff1}
		cannot contain a corner whose apex, say $x_0$, connects to infinity in the unbounded component of $\mathbb{R}^3\setminus\overline{\Omega\cup\Omega'}$, where \eqref{eq:cc11} and $\mathbf{J}_j,\mathbf{J}'_j\in C^\alpha(\mathcal{K}_{x_0}^{r_0})^3$, $j=1,2$, are satisfied with % such that the following two conditions are fulfilled,
%		\begin{equation}\label{eq:cc111}
%			\mathbf{J}_j\in C^\alpha(\mathcal{K}_{x_0}^{r_0})^3\quad\mbox{and}\quad \mathbf{J}'_j\in C^\alpha(\mathcal{K}_{x_0}^{r_0})^3,\quad j=1,2,
%		\end{equation}	
%		and 
%		\begin{equation}\label{eq:cc11}
%			\begin{split}
%				(i)&.~\mbox{$\mathbf{J}_1(x_0)\neq 0$ or $\mathbf{J}_2(x_0)\neq 0$ if $x_0$ is a corner of $\Omega$}, \\ \quad (ii)&.~\mbox{$\mathbf{J}'_1(x_0)\neq 0$ or $\mathbf{J}'_2(x_0)\neq 0$ if $x_0$ is a corner of $\Omega'$},
%			\end{split}
%		\end{equation}
%		where 
		\begin{equation}\label{eq:cc113}
			\mathbf{J}_1:=(\mu-\mu_0)\mathbf{H}_t,\ \mathbf{J}_2:=(\gamma-\varepsilon_0)\mathbf{E}_t;\ \ \mathbf{J}'_1:=(\mu'-\mu_0)\mathbf{H}_t',\ \mathbf{J}'_2:=(\gamma'-\varepsilon_0)\mathbf{E}_t'.
		\end{equation}
	\end{thm}
	\begin{proof}
		The proof follows from a similar argument as that of Theorem~\ref{thm:main3} along with the use of the same reduction strategy in Theorem~\ref{thm:main5} in transforming the medium scattering problem to a source scattering problem.
	\end{proof}
	
	Similarly as before, the %technical 
	regularity assumptions in Theorem~\ref{thm:main7} limit the practical applicability of the local unique recovery result in Theorem~\ref{thm:main7}. Instead of exploring under what conditions the regularity assumption can be fulfilled, we leave this issue for our future study, in particular, on relaxing the regularity condition in Theorem~\ref{thm:main1}.

	\section*{Acknowledgment}
	
	The work of H Liu was supported by the startup fund and FRG grants from Hong Kong Baptist University and the Hong Kong RGC grants (projects 12302017 and 12302018).


\begin{thebibliography}{99}
		
		\bibitem{ABDCvector}
		C. Amrouche, C. Bernardi, M. Dauge, and V. Girault,
		{\it Vector potentials in three-dimensional non-smooth domains},
		Math. Methods Appl. Sci.,
		{\bf 21} (1998), no. 9, 823--864.
		
		\bibitem{Bsource}
		E.~{Bl{\aa}sten}, {\it Nonradiating sources and transmission eigenfunctions vanish at
			corners and edges}, SIAM J. Math. Anal. {\bf 50} (2018), no. 6, 6255--6270.
		
		\bibitem{BLY} E. Bl{\aa}sten and Y.-H. Lin, {\it Radiating and non-radiating sources in elasticity }, Inverse Problems, {\bf 35} (2019), no. 1, 015005.
		
		\bibitem{BL2016}
		E. Bl{\aa}sten and H. Liu,  {\it On corners scattering stably, nearly non-scattering interrogating waves, and stable shape determination by a single far-field pattern},
		Indiana Univ. Math. J., in press, 2019. 
		
		\bibitem{BL2017}
		E. Bl{\aa}sten and H. Liu, {\it Recovering piecewise constant refractive indices by a single far-field pattern}, arXiv:1705.00815, 2017.
		
		\bibitem{BL2017b} 
		E. Bl{\aa}sten and H. Liu, {\it On vanishing near corners of transmission eigenfunctions}, J. Funct. Anal., {\bf 273} (2017), no. 11, 3616--3632. Addendum, arXiv:1710.08089, 2017.
		
		\bibitem{BL2018}
		E. Bl{\aa}sten and H. Liu, {\it Scattering by curvatures, radiationless sources, transmission eigenfunctions and inverse scattering problems}, arXiv:1808.01425, 2018.
		
		\bibitem{BLLW} E. Bl{\aa}sten, X. Li, H. Liu and Y. Wang, {On vanishing and localization near cusps of transmission
			eigenfunctions: a numerical study}, Inverse Problems, {\bf 33} (2017), 105001. 
		
		
		\bibitem{BPS}
		{E. Bl{\aa}sten, L. P\"aiv\"arinta and J. Sylvester},  {\it Corners always scatter},
		Comm.\ Math.\ Phys., {\bf 331} (2014), 725--753.
		
		\bibitem{BV} E. Bl{\aa}sten and E. Vesalainen, {\it Non-scattering energies and transmission eigenvalues in $\mathbb{H}^n$}, arXiv:1809.04426, 2018.
		
		\bibitem{Blei} N. Bleistein and J.K. Cohen, {\it Nonuniqueness in the inverse source problem in
			acoustics and electromagnetics}, J. Math. Phys., {\bf 18} (1977), 194--201.
		
		\bibitem{Boh} D. Bohm and M. Weinstein, {\it The self-oscillations of a charged particle}, Phys.
		Rev., 74:1789--1798, 1948.
		
		
		\bibitem{CHreview}
		F. Cakoni and H. Haddar, {\it Transmission eigenvalues in inverse scattering theory}, in ``Inverse Problems and Applications: Inside Out II'', Math. Sci. Res. Inst. Publ., Vol. 60,  pp. 529--580, Cambridge Univ. Press., Cambridge, 2013. 
		
		\bibitem{CaX}
		F. Cakoni and J. Xiao, {\it On corner scattering for operators of divergence form and applications to inverse scattering},
		arXiv:1905.02558, 2019.
		
		\bibitem{CoK13}
		D.~Colton and R.~Kress, {\it Inverse Acoustic and Electromagnetic Scattering Theory}, 2nd edition, Springer-Verlag, Berlin, 1998.
		
		\bibitem{CKP} D. Colton, A. Kirsch and L. P\"aiv\"arinta, {\it Far-field patterns for acoustic waves in an inhomogeneous medium}, {SIAM J. Math. Anal.}, {\bf 20} (1989),1472--1483. 
		
		\bibitem{CM}
		D. Colton and P. Monk, {\it The inverse scattering problem for time-harmonic acoustic waves in an inhomogeneous medium}, Quart. J. Mech. Appl. Math., {\bf 41} (1988), 97--125. 
		
		\bibitem{CGH}
		A. Cossonni\'ere and H. Haddar, {\it The electromagnetic interior transmission problem for regions with cavities}, SIAM J. Math. Anal., {\bf 43} (2011), no. 4, 1698--1715.
		
		\bibitem{DCL} H. Diao, X. Cao and H. Liu, {\it On the geometric structures of conductive transmission eigenfunctions and their application}, arXiv:1811.01663, 2018.
		
		\bibitem{Deva2} A.J. Devaney and E. Wolf, {\it Non-radiating stochastic scalar sources}, In L. Mandel
		and E. Wolf, editors, Coherence and Quantum Optics V, pages 417--421,
		New York, 1984. Plenum Press.
		
		\bibitem{Ehr} P. Ehrenfest, {\it Ungleichf\"ormige Elektrizit\"atsbewegungen ohne Magnet- und
			Strahlungsfeld}, Physik. Zeit., 11:708--709, 1910.
		
		\bibitem{ElH15} J. Elschner and G. Hu, {\it Corners and edges always scatter}, Inverse Problems {\bf 31} (2015), no. 1, 015003.
		
		\bibitem{ElH} J. Elschner and G. Hu, {\it Acoustic scattering from corners, edges and circular cones}, Arch. Ration. Mech. Anal., {\bf 228} (2018), no. 2, 653--690.
		
		\bibitem{Frie} F.G. Friedlander, {\it An inverse problem for radiation fields}, Proc. London Math.
		Soc., {\bf 2} (1973), 551--576, 1973.
		
		\bibitem{Gam} A. Gamliel, K. Kim, A.I. Nachman, and E. Wolf,  {\it A new method for specifying
			nonradiating monochromatic sources and their fields}, J. Opt. Soc. Am. A,
		{\bf 6} (1989), 1388--1393.
		
		\bibitem{Gbu} G. Gbur, {\it Nonradiating Sources and the Inverse Source Problem}, PhD Thesis, Univ. Rochester, 2001.
		
		\bibitem{Goe} G.H. Goedecke, {\it Classically radiationless motions and possible implications for
			quantum theory}, Phys. Rev., 135:B281--B288, 1964.
		
		\bibitem{GKLU4} {A.~Greenleaf, Y.~Kurylev, M.~Lassas and G.~Uhlmann}, {\it Invisibility and inverse problems}, Bulletin A. M. S., {\bf
			46} (2009), 55--97.
		
		\bibitem{GKLU5} {A.~Greenleaf, Y.~Kurylev, M.~Lassas and G.~Uhlmann}, {\it Cloaking devices, electromagnetic wormholes and
			transformation optics}, SIAM Review, {\bf 51} (2009), 3--33.
		
		\bibitem{Hoen} B.J. Hoenders and H.P. Baltes, {\it The scalar theory of nonradiating partially
			coherent sources}, Lettere al Nuovo Cimento, {\bf 23} (1979), 206--208.
		
		\bibitem{HSV} G. Hu, M. Salo and E. Vesalainen, {\it Shape identification in inverse medium scattering problems with a single far-field pattern}, SIAM J. Math. Anal. {\bf 48} (2016), no. 1, 152--165.
		
		\bibitem{Ikehata} M. Ikehata, {\it Reconstruction of a source domain from the {C}auchy data}, {Inverse Problems}, {\bf 15} (1999), 637--645.
		
		\bibitem{Kim} K. Kim and E. Wolf, {\it Non-radiating monochromatic sources and their fields},
		Opt. Commun., {\bf 59} (1986), 1--6.
		
		\bibitem{Kir} A. Kirsch, {\it The denseness of the far field patterns for the transmission problem}, IMA J. Appl. Math., {\bf 37} (1986), 213--225.
		
		\bibitem{KS1} {S. Kusiak and J. Sylvester}, {\it The scattering support}, Comm. Pure Appl. Math., {\bf 56} (2003), 1525--1548. 
		
		\bibitem{KS2} {S. Kusiak and J. Sylvester}, {\it The convex scattering support in a background medium}, SIAM J. Math. Anal., {\bf 36} (2005), 1142--1158. 
		
		\bibitem{LLL17} J. Li, X. Li and H. Liu, {\it Reconstruction via the intrinsic geometric structures of interior transmission eigenfunctions}, arXiv:1706.04418, 2017.
		
		\bibitem{LRX16}
		{ H.~Liu, L.~Rondi, and J.~Xiao}, {\em {Mosco} convergence for
			{$H(\mathrm{curl})$} spaces, higher integrability for {Maxwell}'s equations,
			and stability in direct and inverse {EM} scattering problems}, J. Eur. Math. Soc., {\bf 21} (2019), 2945--2993.
%		\newblock to appear, Journal of the European Mathematical Society,
%		arXiv:1603.07555, 2016.
		
		\bibitem{LiuUhl} {H. Liu and G. Uhmann}, {\it Regularized transformation-optics cloaking in acoustic and electromagnetic scattering}, Inverse problems and imaging, 111--136, Panor. Synth\'eses, 44, Soc. Math. France, Paris, 2015.
		
		\bibitem{LX} H. Liu and J. Xiao, {\it On electromagnetic scattering from a penetrable corner}, SIAM J. Math. Anal., {\bf 49} (2017), no. 6, 5207--5241.
		
		\bibitem{Mare} E.A. Marengo and R.W. Ziolkowski, {\it On the radiating and nonradiating components
			of scalar, electromagnetic, and weak gravitational sources}, Phys. Rev.
		Lett., {\bf 83} (1999), 3345--3349.
		
		\bibitem{Ned01}
		{\sc J.-C. N{\'e}d{\'e}lec}, {\em Acoustic and Electromagnetic Equations:
			Integral Representations for Harmonic Problems}, vol.~144 of Applied
		Mathematical Sciences, Springer-Verlag New York, 2001.
		
		\bibitem{PS} L. P\"aiv\"arinta and J. Sylvester, {\it Transmission eigenvalues}, {SIAM J. Math. Anal.}, {\bf 40} (2008), 738--753.
		
		\bibitem{PSV} {L.~P\"aiv\"arinta, M.~Salo and E.~V.~Vesalainen}, {\it Strictly convex corners scatter}, Rev. Mat. Iberoam. {\bf 33} (2017), no.~4, 1369--1396.
		
		\bibitem{RS} B.~P. Rynne and B.~D. Sleeman, {\it The interior transmission problem and inverse scattering from inhomogeneous media}, {SIAM J. Math. Anal.}, {\bf 22} (1991), 1755--1762. 
		
		\bibitem{Som1} A. Sommerfeld, {\it Zur elektronentheorie. I. Allgemeine untersuchung des feldes
			eines beliebig bewegten elektrons}, Akad. der Wiss. (G\"ott.), Math. Phys. Klasse,
		Nach., 99--130, 1904.
		
		\bibitem{Som2} A. Sommerfeld, {\it Zur elektronentheorie. II. Grundlagen f\"ur eine allgemeine dynamik
			des elektrons}, Akad. der Wiss. (G\"ott.), Math. Phys. Klasse, Nach., 
		363--439, 1904.
		
	\end{thebibliography}
\end{document}